\documentclass[12pt]{amsart}
\usepackage{amssymb}
\usepackage{hyperref}
\usepackage{color}

\theoremstyle{plain}
\newtheorem{theorem}{Theorem}[section]

\newtheorem{corollary}[theorem]{Corollary}
\newtheorem{lemma}[theorem]{Lemma}

\theoremstyle{definition}
\newtheorem{definition}[theorem]{Definition}

\newtheorem{remark}[theorem]{Remark}

\newtheorem{notation}[theorem]{Notation}

\theoremstyle{remark}

\newcommand{\N}{\mathbb N}
\newcommand{\Z}{\mathbb Z}

\newcommand{\R}{\mathbb R}
\newcommand{\C}{\mathbb C}

\newcommand{\GL}{\operatorname{GL}}
\newcommand{\SL}{\operatorname{SL}}

\newcommand{\SO}{\operatorname{SO}}
\newcommand{\SU}{\operatorname{SU}}
\newcommand{\U}{\operatorname{U}}
\newcommand{\Sp}{\operatorname{Sp}}
\newcommand{\Spin}{\operatorname{Spin}}

\newcommand{\gl}{\mathfrak{gl}}
\newcommand{\sll}{\mathfrak{sl}}
\newcommand{\so}{\mathfrak{so}}
\newcommand{\spp}{\mathfrak{sp}}

\newcommand{\op}{\operatorname}

\newcommand{\diag}{\operatorname{diag}}

\newcommand{\ba}{\backslash}

\newcommand{\mi}{\mathtt{i}}
\newcommand{\norma}[1]{\|{#1}\|_1}
\newcommand{\contador}{Z}
\newcommand{\tipo}{\mathrm}

\title[Multiplicity formulas]{Multiplicity formulas for fundamental strings of representations of classical Lie algebras}
\author{Emilio~A.~Lauret, Fiorela Rossi Bertone}
\address{L.: Institut f\"ur Mathematik, Humboldt Universit\"at zu Berlin, Unter den Linden 6, 10099 Berlin, Germany. 
	Permanent affiliation: CIEM--FaMAF (CONICET), Universidad Nacional de C\'ordoba, Medina Allende s/n, Ciudad Universitaria, 5000 C\'ordoba, Argentina.}
\email{elauret@famaf.unc.edu.ar}
\address{R.B.: CIEM--FaMAF (CONICET), Universidad Nacional de C\'ordoba, Medina Allende s/n, Ciudad Universitaria, 5000 C\'ordoba, Argentina.}
\email{rossib@famaf.unc.edu.ar}
\thanks{This research was partially supported by grants from CONICET, FONCyT and SeCyT--UNC. The first named author was supported by the Alexander von Humboldt Foundation}
\date{Noviembre 2017}

\begin{document}

\begin{abstract}
We call the \emph{$p$-fundamental string} of a complex simple Lie algebra to the sequence of irreducible representations having highest weights of the form $k\omega_1+\omega_p$ for $k\geq0$, where $\omega_j$ denotes the $j$-th fundamental weight of the associated root system.  
For a classical complex Lie algebra, we establish a closed explicit formula for the weight multiplicities of any representation in any $p$-fundamental string. 
\end{abstract}

\maketitle

\tableofcontents

\section{Introduction}\label{sec:intro}

Let $\mathfrak g$ be a complex semisimple Lie algebra.
We fix a Cartan subalgebra $\mathfrak h$ of $\mathfrak g$. 
Let $(\pi,V_\pi)$ be a finite dimensional representation of $\mathfrak g$, that is, a homomorphism $\pi:\mathfrak g\to \gl(V_\pi)$ with $V_\pi$ a complex vector space.  
An element $\mu\in \mathfrak h^*$ is called a \emph{weight} of $\pi$ if 
\begin{equation*}
V_\pi(\mu):= \{v\in V_\pi: \pi(X)v=\mu(X)v \text{ for all }X\in\mathfrak h\}\neq0. 
\end{equation*}
The \emph{multiplicity} of $\mu$ in the representation $\pi$, denoted by $m_\pi(\mu)$, is defined as $\dim  V_\pi(\mu)$. 

There are many formulas in the literature to compute $m_\pi(\mu)$ for arbitrary $\mathfrak g$, $\pi$ and $\mu$.
The ones by Freudenthal~\cite{Freudenthal54} and Kostant~\cite{Kostant59} are very classical. 
More recent formulas were given by Lusztig~\cite{Lusztig83}, Littelmann~\cite{Littelmann95} and Sahi~\cite{Sahi00}.
Although all of them are very elegant and powerful theoretical results, they may not be considered \emph{closed explicit expressions}.
Moreover, some of them are not adequate for computer implementation (cf.\ \cite{Schutzer04thesis}, \cite{Harris12thesis}).

Actually, it is not expected a closed formula in general. 
There should always be a sum over a symmetric group (whose cardinal grows quickly when the rank of $\mathfrak g$ does) or over partitions, or being recursive, or written in terms of combinatorial objects (e.g.\ Young diagrams like in \cite{Koike87}), among other ways.

However, closed explicit expressions are possible for particular choices of $\mathfrak g$ and $\pi$.  
Obviously, this is the case for $\sll(2,\C)$ and $\pi$ any of its irreducible representations (see \cite[\S I.9]{Knapp-book-beyond}). 
Furthermore, for a classical Lie algebra $\mathfrak g$, it is not difficult to give expressions for the weight multiplicities of the representations $\op{Sym}^k(V_\mathrm{st})$ and $\bigwedge^p (V_\mathrm{st})$ and also for their irreducible components (see for instance Lemmas~\ref{lemCn:extreme}, \ref{lemDn:extremereps} and \ref{lemBn:extremereps} and Theorem~\ref{thmAn:multip(k,p)}; these formulas are probably well known but they are included here for completeness).
Here, $V_{\mathrm{st}}$ denotes the standard representation of $\mathfrak g$. 
A good example of a closed explicit formula in a non-trivial case was given by Cagliero and Tirao~\cite{CaglieroTirao04} for $\spp(2,\C)\simeq\so(5,\C)$ and $\pi$ arbitrary.

In order to end the description of previous results in this large area we name a few recent related results, though the list is far from being complete: \cite{Cochet05}, \cite{BaldoniBeckCochetVergne06}, \cite{Bliem08-thesis}, \cite{Schutzer12}, \cite{Maddox14},  \cite{FernandezGarciaPerelomov2014}, \cite{FernandezGarciaPerelomov2015a}, \cite{FernandezGarciaPerelomov2015b}, \cite{FernandezGarciaPerelomov2017},  \cite{Cavallin17}.

The main goal of this article is to show, for each classical complex Lie algebra $\mathfrak g$ of rank $n$, a closed explicit formula for the weight multiplicities of any irreducible representation of $\mathfrak g$ having highest weight $k\omega_1+\omega_p$, for any integers $k\geq0$ and $1\leq p\leq n$.
Here, $\omega_1,\dots,\omega_n$ denote the fundamental weights associated to the root system $\Sigma(\mathfrak g,\mathfrak h)$.
We call \emph{$p$-fundamental string} to the sequence of irreducible representations of $\mathfrak g$ with highest weights $k\omega_1+\omega_p$ for $k\geq0$. 
We will write $\pi_{\lambda}$ for the irreducible representation of $\mathfrak g$ with highest weight $\lambda$.

For types $\tipo B_n$, $\tipo C_n$ or $\tipo D_n$  (i.e.\ $\so(2n+1,\C)$, $\spp(n,\C)$ or $\so(2n,\C)$ respectively) an accessory representation $\pi_{k,p}$ is introduced to unify the approach (see Definition~\ref{def:pi_kp}). 
We have that $\pi_{k,p}$ and $\pi_{k \omega_1+\omega_p}$ coincide except for $p=n$ in type $\tipo B_n$ and $p=n-1,n$ in type $\tipo D_n$. 
The weight multiplicity formulas for $\pi_{k,p}$ are in Theorems~\ref{thmCn:multip(k,p)}, \ref{thmDn:multip(k,p)} and \ref{thmBn:multip(k,p)} for types $\tipo C_n$, $\tipo D_n$ and $\tipo B_n$ respectively.
Their proofs follow the same strategy (see Section~\ref{sec:strategy}). 
The formulas for the remaining cases, namely the (spin) representations $\pi_{k \omega_1+\omega_n}$ in type $\tipo B_n$ and $\pi_{k \omega_1+\omega_{n-1}}$, $\pi_{k \omega_1+\omega_n}$ in type $\tipo D_n$, can be found in Theorem~\ref{thmDn:multip(spin)} and \ref{thmBn:multip(spin)} respectively.

Given a weight $\mu=\sum_{j=1}^{n} a_j\varepsilon_j$ (see Notation~\ref{notacion}) of a classical Lie algebra $\mathfrak g$ of type $\tipo B_n$, $\tipo C_n$ or $\tipo D_n$, we set
\begin{align}\label{eq:notation-one-norm}
\norma{\mu}=\sum_{j=1}^{n} |a_j| 
\quad\text{and}\quad 
\contador (\mu) = \#\{1\leq j\leq n: a_j=0\}. 
\end{align} 
We call $\norma{\mu}$ the \emph{one-norm} of $\mu$.
The function $\contador(\mu)$ counts the number of zero coordinates of $\mu$.
It is not difficult to check that $m_{\pi_{k\omega_1}}(\mu)$ depends only on $\norma{\mu}$ for a fixed $k\geq0$.
Moreover, it is known that $m_{\pi_{k,p}}(\mu)$ depends only on $\norma{\mu}$ and $\contador (\mu)$ for type $\tipo D_n$ (see \cite[Lem.~3.3]{LMR-onenorm}).
This last property is extended to types $\tipo B_n$ and $\tipo C_n$ as a consequence of their multiplicity formulas.  

\begin{corollary}\label{cor:depending-one-norm-ceros}
For $\mathfrak g$ a classical Lie algebra of type $\tipo B_n$, $\tipo C_n$ or $\tipo D_n$ and a weight $\mu=\sum_{i=1}^{n} a_i\varepsilon_i$, the multiplicity of $\mu$ in $\pi_{k,p}$ depends only on $\norma{\mu}$ and $\contador (\mu)$. 
\end{corollary}

For $\mathfrak g=\sll(n+1,\C)$ (type $\tipo A_n$), the multiplicity formula for a representation in a fundamental string is in Theorem~\ref{thmAn:multip(k,p)}. 
This case is simpler since it follows immediately from basic facts on Young diagrams.
Although this formula should be well known, it is included for completeness.  

Explicit expressions for the weight multiplicities of a representation in a fundamental string are required in several different areas. 
The interest of the authors on them comes from their application to spectral geometry.
Actually, many multiplicity formulas have already been applied to determine the spectrum of Laplace and Dirac operators on certain locally homogeneous spaces. 
See Section~\ref{sec:conclusions} for a detailed account of these applications. 

It is important to note that all the weight multiplicity formulas obtained in this article have been checked with Sage~\cite{Sage} for many cases.
This computer program uses the classical Freudenthal formula. 
Because of the simplicity of the expressions obtained in the main theorems, the computer takes usually a fraction of a second to calculate the result. 

Throughout the article we use the convention $\binom{b}{a}=0$ if $a<0$ or $b<a$.

The article is organized as follows. 
Section~\ref{sec:strategy} explains the method to obtain $m_{\pi_{k,p}}(\mu)$ for types $\tipo B_n$, $\tipo C_n$ and $\tipo D_n$. 
These cases are considered in Sections~\ref{secBn:multip(k,p)}, \ref{secCn:multip(k,p)} and \ref{secDn:multip(k,p)} respectively, and type $\tipo A_n$ is in Section~\ref{secAn:multip(k,p)}. 
In Section~\ref{sec:conclusions} we include some conclusions.

\section{Strategy}\label{sec:strategy}

In this section, we introduce the abstract method used to find the weight multiplicity formulas for the cases $\tipo B_n$, $\tipo C_n$ and $\tipo D_n$. 
Throughout this section, $\mathfrak g$ denotes a classical complex Lie algebra of type $\tipo B_n$, $\tipo C_n$ and $\tipo D_n$, namely $\so(2n+1,\C)$, $\spp(n,\C)$, $\so(2n,\C)$, for some $n\geq2$.
We first introduce some standard notation.

\begin{notation}\label{notacion}
We fix a Cartan subalgebra $\mathfrak h$ of $\mathfrak g$.  
Let $\{\varepsilon_{1}, \dots,\varepsilon_{n}\}$ be the standard basis of $\mathfrak h^*$. Thus, the sets of simple roots $\Pi(\mathfrak g,\mathfrak h)$ are given by $\{\varepsilon_1-\varepsilon_2,\dots, \varepsilon_{n-1}-\varepsilon_n,\varepsilon_n\}$ for type $\tipo B_n$, $\{\varepsilon_1-\varepsilon_2,\dots, \varepsilon_{n-1}-\varepsilon_n,2\varepsilon_n\}$ for type $\tipo C_n$, and $\{\varepsilon_1-\varepsilon_2,\dots, \varepsilon_{n-1}-\varepsilon_n,\varepsilon_{n-1}+\varepsilon_n\}$ for type $\tipo D_n$. 
A precise choice for $\mathfrak h$ and $\varepsilon_j$ will be indicated in each type. 

We denote by $\Sigma(\mathfrak g,\mathfrak h)$ the set of roots, by $\Sigma^+(\mathfrak g,\mathfrak h)$ the set of positive roots, by $\omega_1,\dots,\omega_n$ the fundamental weights, by $P(\mathfrak g)$ the (integral) weight space of $\mathfrak g$ and by $P^{{+}{+}}(\mathfrak g)$ the set of dominant weights.

Let $\mathfrak g_0$ be the compact real form of $\mathfrak g$ associated to $\Sigma(\mathfrak g,\mathfrak h)$, let $G$ be the compact linear group with Lie algebra $\mathfrak g_0$ (e.g.\ $G=\SO(2n)$ for type $\tipo D_n$ in place of $\Spin(2n)$), and let $T$ be the maximal torus in $G$ corresponding to $\mathfrak h$, that is, the Lie algebra $\mathfrak t$ of $T$ is a real subalgebra of $\mathfrak h$.
Write $P(G)$ for the set of $G$-integral weights and $P^{{+}{+}}(G)=P(G)\cap P^{{+}{+}}(\mathfrak g)$.

By the Highest Weight Theorem, the irreducible representations of $\mathfrak g$ and $G$ are in correspondence with elements in $P^{{+}{+}}(\mathfrak g)$ and $P^{{+}{+}}(G)$ respectively.
For $\lambda$ an integral dominant weight, we denote by $\pi_\lambda$ the associated irreducible representation of $\mathfrak g$. 
\end{notation}

We recall that, under Notation~\ref{notacion}, the fundamental weights are: 
\begin{align*}
\text{in type $\tipo B_n$},\qquad 
\omega_p &=
\begin{cases}
	\varepsilon_1+\dots+\varepsilon_p 
	&\text{if $1\leq p\leq n-1$,}\\
	\frac12(\varepsilon_1+\dots+\varepsilon_n) 
	&\text{if $p=n$,}
\end{cases} \\
\text{in type $\tipo C_n$},\qquad
\omega_p &=\varepsilon_1+\dots+\varepsilon_p \quad\text{for every $1\leq p\leq n$},\\
\text{in type $\tipo D_n$},\qquad 
\omega_p &=
\begin{cases}
	\varepsilon_1+\dots+\varepsilon_p 
	&\text{if $1\leq p\leq n-2$,}\\
	\frac12(\varepsilon_1+\dots+\varepsilon_{n-1}-\varepsilon_{n}) 
	&\text{if $p=n-1$,}\\
	\frac12(\varepsilon_1+\dots+\varepsilon_{n-1}+\varepsilon_{n}) 
	&\text{if $p=n$.}
\end{cases}
\end{align*}
We set $\widetilde \omega_p=\varepsilon_1+\dots+\varepsilon_p$ for any $1\leq p\leq n$. 
Thus, $\widetilde \omega_p=\omega_p$ excepts for type $\tipo B_n$ and $p=n$ when $\widetilde \omega_{n}=2\omega_n$, and for type $\tipo D_n$ and $p\in\{n-1,n\}$ when $\widetilde \omega_{n-1}=\omega_{n-1}+\omega_n$ and $\widetilde \omega_{n}=2\omega_n$.

\begin{definition}\label{def:pi_kp}
Let $\mathfrak g$ be a classical Lie algebra of type $\tipo B_n$, $\tipo C_n$ or $\tipo D_n$. 
For $k\geq0$ and $1\leq p\leq n$ integers, let us denote by $\pi_{k,p}$ the irreducible representation of $\mathfrak g$ with highest weight $k\omega_1+\widetilde \omega_p$, except for $p=n$ and type $\tipo D_n$ when we set $\pi_{k,n}=\pi_{k\omega_1+2\omega_{n-1}}\oplus \pi_{k\omega_1+2\omega_{n}}$.
By convention, we set $\pi_{k,0}=0$ for $k\geq0$. 
\end{definition}

We next explain the procedure to determine the multiplicity formula for $\pi_{k,p}$.

\begin{description}
	\item[Step 1]
	Obtain the decomposition in irreducible representations of 
	\begin{equation}\label{eq:sigma_kp}
	\sigma_{k,p}:=\pi_{k\omega_1}\otimes \pi_{\widetilde \omega_p},
	\end{equation}
	and consequently, write $\pi_{k,p}$ in terms of representations of the form \eqref{eq:sigma_kp} in the virtual representation ring. 
	
	Fortunately, this decomposition is already known and coincides for the  types $\tipo B_n$, $\tipo C_n$ and $\tipo D_n$, thus the second requirement has also a uniform statement (see Lemma~\ref{lem:step1}).

	\item[Step 2]
	Obtain a formula for the weight multiplicities of the extreme cases $\pi_{k\omega_1}$ and $\pi_{\widetilde \omega_p}$.

	It will be useful to realize these representations inside $\op{Sym}^k(V_{\pi_{\omega_1}})$ and $\bigwedge^p(V_{\pi_{\omega_1}})$ respectively. 
	Note that $\pi_{\omega_1}$ is the standard representation.

	\item[Step 3]
	Obtain a closed expression for the weight multiplicities on $\sigma_{k,p}$.
	 
	This is the hardest step. 
	One has that (see for instance \cite[Exercise~V.14]{Knapp-book-beyond})
	\begin{equation}\label{eq:multiptensor}
	m_{\sigma_{k,p}}(\mu) =
	\sum_{\eta} m_{\pi_{k\omega_1}}(\mu-\eta) \, m_{\pi_{\widetilde \omega_p}}(\eta),
	\end{equation}
	where the sum is over the weights of $\pi_{\widetilde \omega_p}$.
	Then, the multiplicity formulas obtained in Step~2 can be applied. 
	
	\item[Step 4]
	Obtain the weight multiplicity formula for $\pi_{k,p}$.
	
	We will replace the formula obtained in Step~3 into the formula obtained in Step~1.
\end{description}

\smallskip

The following result works out Step~1.

\begin{lemma}\label{lem:step1}
	Let $\mathfrak g$ be a classical Lie algebra of type $\tipo B_n$, $\tipo C_n$ or $\tipo D_n$ and let $k\geq0$, $1\leq p\leq n$ integers.
	Then
	\begin{equation}\label{eq:funsionrule(sigma)}
	\sigma_{k,p} = \pi_{k\omega_1}\otimes \pi_{\widetilde \omega_p}  = 
	\pi_{k-1,1}\otimes \pi_{0,p} \simeq \pi_{k,p}\oplus \pi_{k-1,p+1} \oplus \pi_{k-2,p}\oplus \pi_{k-1,p-1}.
	\end{equation} 
	Furthermore, in the virtual ring of representations, we have that 
	\begin{equation}\label{eq:virtualring(sigma)}
	\pi_{k,p} = \sum_{j=1}^p (-1)^{j-1} \sum_{i=0}^{j-1}  \sigma_{k+j-2i,p-j}.
	\end{equation}
\end{lemma}

\begin{proof}
	The decomposition \eqref{eq:funsionrule(sigma)} is proved in \cite[page 510, example (3)]{KoikeTerada87} by Koike and Terada, though their results are much more general and this particular case was probably already known.  
	
We now show \eqref{eq:virtualring(sigma)}.
The case $p=1$ is trivial.
Indeed, the right hand side equals $\sigma_{k+1,0}=\pi_{k,1}$ by definition.
We assume that the formula is valid for values lower than or equal to $p$.
By this assumption and \eqref{eq:funsionrule(sigma)} we have that
\begin{align*}
	\pi_{k,p+1}
	&= \sigma_{k+1,p} - \pi_{k+1,p} -\pi_{k-1,p} -\pi_{k,p-1} 
	= \sigma_{k+1,p} - \sum_{j=1}^p (-1)^{j-1}\sum_{i=0}^{j-1} \sigma_{k+1+j-2i,p-j}\\
	&\qquad      - \sum_{j=1}^p (-1)^{j-1}\sum_{i=0}^{j-1} \sigma_{k-1+j-2i,p-j}
	- \sum_{j=1}^{p-1} (-1)^{j-1}\sum_{i=0}^{j-1} \sigma_{k+j-2i,p-1-j}.
\end{align*}
By making the change of variables $h=j+1$ in the last term, one gets 
\begin{align*}
\pi_{k,p+1}
	&= \sigma_{k+1,p}- \sum_{j=1}^p (-1)^{j-1}\sum_{i=0}^{j-1} \sigma_{k+1+j-2i,p-j}- \sigma_{k,p-1}  - \sum_{j=2}^p (-1)^{j-1} \sigma_{k+1-j,p-j}. 
\end{align*}
The rest of the proof is straightforward. 
\end{proof}

\section{Type C} \label{secCn:multip(k,p)}

In this section we consider the classical Lie algebra $\mathfrak g$ of type $\tipo C_n$, that is, $\mathfrak g=\spp(n,\C)$.
In this case, according to Notation~\ref{notacion}, $\widetilde \omega_p=\omega_p$ for every $p$, thus $\pi_{k\omega_1+\omega_p} = \pi_{k,p}$. 
The next theorem gives the explicit expression of $m_{\pi_{k,p}}(\mu)$ for any weight $\mu$. 
This expression depends on the terms $\norma{\mu}$ and $\contador (\mu)$, introduced in \eqref{eq:notation-one-norm}.

\begin{theorem}\label{thmCn:multip(k,p)}
	Let $\mathfrak g=\spp(n,\C)$ for some $n\geq2$ and let $k\geq0$, $1\leq p\leq n$ integers.
	For $\mu\in P(\mathfrak g)$, if $r(\mu):=(k+p-\norma{\mu})/2$ is a non-negative integer, then
	\begin{align*}
		m_{\pi_{k,p}}(\mu)
		&= \sum_{j=1}^{p} (-1)^{j-1}  \sum_{t=0}^{\lfloor\frac{p-j}{2}\rfloor} \frac{n-p+j+1}{n-p+j+t+1}\binom{n-p+j+2t}{t} \\
		&\quad \sum_{\beta=0}^{p-j-2t} 2^{p-j-2t-\beta} \binom{n-\contador (\mu)}{\beta} \binom{\contador (\mu)}{p-j-2t-\beta}  \\
		&\quad \sum_{\alpha=0}^\beta \binom{\beta}{\alpha} \sum_{i=0}^{j-1} \binom{r(\mu)-i-p+\alpha+t+j+n-1}{n-1},		
	\end{align*}
	and $m_{\pi_{k,p}}(\mu)=0$ otherwise.
\end{theorem}

The rest of this section is devoted to prove this formula following the procedure described in Section~\ref{sec:strategy}. 
We first set the notation for this case. 
Here $G=\Sp(n,\C)\cap \U(2n)$ where $\Sp(n,\C) = \{g\in\SL(2n,\C): g^t J_ng=J_n:=\left(\begin{smallmatrix}0&\op{Id}_n\\ -\op{Id}_n&0\end{smallmatrix}\right)\}$,
$\mathfrak g_0=\mathfrak{sp}(n,\C)\cap \mathfrak{u}(2n)$, 
\begin{align}\label{eqCn:maximaltorus}
	T&=
	\left\{
	\diag\left(e^{\mi \theta_1},\dots, e^{\mi \theta_n},e^{-\mi \theta_1},\dots, e^{-\mi \theta_n}
	\right)
	:\theta_i\in\R\;\forall\,i
	\right\},\\
	\label{eqCn:subalgCartan}
	\mathfrak h &=
	\left\{
	\diag(\theta_1,\dots,\theta_n,-\theta_1,\dots,-\theta_n):
	\theta_i\in\C \;\forall\,i
	\right\},
\end{align}
$\varepsilon_i\big(\diag\left(\theta_1,\dots,\theta_n,-\theta_1,\dots,-\theta_n \right)\big) =\theta_i$ for each $1\leq i\leq n$, $\Sigma^+(\mathfrak g,\mathfrak h)= \{\varepsilon_i\pm \varepsilon_j: 1\leq i<j\leq n\}\cup\{2\varepsilon_i:1\leq i\leq n\}$, and
\begin{align*}
	P(\mathfrak g) &= P(G)= \Z\varepsilon_1\oplus\dots\oplus\Z\varepsilon_{n},\\
	P^{{+}{+}}(\mathfrak g) &= P^{{+}{+}}(G)= \left\{\textstyle\sum_{i}a_i\varepsilon_i \in P(\mathfrak g) :a_1\geq a_2\geq  \dots \geq a_{n}\geq0\right\}.
\end{align*}
The following well known identities (see for instance \cite[\S17.2]{FultonHarris-book}) will be useful to show Step~2,
\begin{align}\label{eqCn:extremereps}
\pi_{k\omega_1}=\pi_{k\varepsilon_1} &\simeq  \op{Sym}^k(\C^{2n}),&
\textstyle\bigwedge^p(\C^{2n}) &\simeq \pi_{\omega_p}\oplus \textstyle\bigwedge^{p-2}(\C^{2n}),
\end{align} 
for any integers $k\geq0$ and $1\leq p\leq n$.
Here, $\C^{2n}$ denotes the standard representation of $\mathfrak g=\spp(2n,\C)$.
Since $G=\Sp(n)$ is simply connected, $\pi_{\lambda}$ descends to a representation of $G$ for any $\lambda\in P^{{+}{+}}(\mathfrak g)$. 
In what follows we will work with representations of $G$ for simplicity. 
Thus,
$%\begin{equation}
m_{\pi}(\mu) = \dim \{v\in V_\pi : \pi (\exp X) v = e^{\mu(X)}v\quad \forall\, X\in\mathfrak t\}.
$%\end{equation}

\begin{lemma}\label{lemCn:extreme}
	Let $n\geq2$, $\mathfrak g=\spp(n,\C)$, $k\geq0$, $1\leq p\leq n$ and $\mu=\sum_{j=1}^n a_j\varepsilon_j\in P(\mathfrak g)$. Then
	\begin{align}\label{eqCn:multip(k)}
		m_{\pi_{k\omega_1}}(\mu) &=m_{\pi_{k\varepsilon_1}}(\mu)=
		\begin{cases}
			\binom{r(\mu)+n-1}{n-1} & \text{ if }\, r(\mu):=\frac{k-\norma{\mu}}{2}\in \N_0,\\
			0 & \text{ otherwise,}
		\end{cases}
		 \\
		m_{\pi_{\omega_p}}(\mu) &=
		\begin{cases}
			\frac{n-p+1}{n-p+r(\mu)+1}\binom{n-p+2r(\mu)}{r(\mu)} & \text{if }\,r(\mu):=\frac{p-\norma{\mu}}{2}\in \N_0  \text{ and }  |a_j|\leq1\;\forall\,j,\\
			0&\text{otherwise.}
		\end{cases}
		\label{eqCn:multip(p)}
	\end{align}
\end{lemma}

\begin{proof}
By \eqref{eqCn:extremereps}, $\pi_{k\varepsilon_1}$ is realized in the space of homogeneous polynomials $\mathcal P_k\simeq \op{Sym}^k(\C^{2n})$ of degree $k$ in the variables $x_1,\dots,x_{2n}$. 
The action of $g\in G$ on $f(x)\in \mathcal P_k$ is given by $(\pi_{k\varepsilon_1}(g)\cdot f)(x) = f(g^{-1}x)$, where $x$ denotes the column vector $(x_1,\dots,x_{2n})^t$.

The monomials $x_1^{k_1}\dots x_n^{k_n}x_{n+1}^{l_1}\dots x_{2n}^{l_n}$ with $k_1,\dots,k_n,l_1,\dots,l_n$ non-negative integers satisfying that $\sum_{j=1}^{n} k_j+l_j=k$ form a basis of $\mathcal P_k$ given by weight vectors.
Indeed, one can check that the action of $h=\diag\left(e^{\mi \theta_1},\dots, e^{\mi \theta_n},e^{-\mi \theta_1},\dots, e^{-\mi \theta_n} \right) \in T$ on the monomial $x_1^{k_1}\dots x_n^{k_n}x_{n+1}^{l_1}\dots x_{2n}^{l_n}$ is given by multiplication by
$
e^{\mi\sum_{j=1}^n\theta_j(k_j-l_j)}.
$ 
Hence, $x_1^{k_1}\dots x_n^{k_n}x_{n+1}^{l_1}\dots x_{2n}^{l_n}$ is a weight vector of weight  $\mu=\sum_{j=1}^n (k_j-l_{j}) \varepsilon_j$.

Consequently, the multiplicity of a weight $\mu=\sum_{j=1}^n a_j\varepsilon_j\in\mathcal P(\mathfrak g)$ in $\mathcal P_k$ is the number of different tuples $(k_1,\dots,k_{n},l_1,\dots,l_{n})\in\N_0^{2n}$ satisfying that $\sum_{j=1}^{n} (k_j+ l_j)=k$ and $a_j=k_j-l_{j}$ for all $j$. 
For such a tuple, we note that $k-\norma{\mu}= k-\sum_{i=1}^n |a_i|=2\sum_{i=1}^n \operatorname{min}(k_i,l_i)$. 
It follows that $\mu$ is a weight of $\mathcal{P}_k$ if and only if $k-\norma{\mu}=2r$ with $r$ a non-negative integer.
Moreover, its multiplicity is the number of different ways one can write $r$ as an ordered sum of $n$ non-negative integers, which equals $\binom{r+n-1}{n-1}$.
This implies \eqref{eqCn:multip(k)}.

For \eqref{eqCn:multip(p)}, we consider the representation $\bigwedge^p(\C^{2n})$.
The action of $G$ on $\bigwedge^p(\C^{2n})$ is given by 
$
g\cdot v_1\wedge\dots\wedge v_p = (g v_1)\wedge\dots\wedge (g v_p),
$
where $gv$ stands for the matrix multiplication between $g\in G\subset \GL(2n,\C)$ and the column vector $v\in \C^{2n}$.

Let $\{e_1,\dots,e_{2n}\}$ denote the canonical basis of $\C^{2n}$. 
For $I=\{i_1,\dots,i_p\}$ with $1\leq i_1<\dots<i_p\leq 2n$, we write $w_I=e_{i_1}\wedge\dots\wedge e_{i_p}$.
Clearly, the set of $w_I$ for all choices of $I$ is a basis of $\bigwedge^p(\C^{2n})$.
Since $h=\diag\left(e^{\mi \theta_1},\dots, e^{\mi \theta_n} ,e^{-\mi \theta_1},\dots, e^{-\mi \theta_n} \right) \in T$ satisfies $h e_j = e^{\mi \theta_j} e_j$ and $h e_{j+n} = e^{-\mi \theta_j}e_{j+n}$ for all $1\leq j\leq n$, we see that $w_I$ is a weight vector of weight $\mu=\sum_{j=1}^n a_j\varepsilon_j$ where
	\begin{align}\label{eq:weight_exteriorCn}
	a_j=\begin{cases}
	1&\quad\text{if $j\in I$ and $j+n\notin I$,}\\
	-1&\quad\text{if $j\notin I$ and $j+n\in I$,}\\
	0&\quad\text{if $j,j+n\in I$ or $j,j+n\notin I$.}
	\end{cases}
	\end{align}
Thus, an arbitrary element $\mu=\sum_j a_j\varepsilon_j\in P(\mathfrak g)$ is a weight of $\bigwedge^p(\C^{2n})$ if and only if $|a_j|\leq 1$ for all $j$ and $p-\norma{\mu}= 2r$ for some non-negative integer $r$.

It remains to determine the multiplicity in $\bigwedge^p(\C^{2n})$ of a weight $\mu=\sum_{j=1}^n a_j\varepsilon_j\in P(\mathfrak{g})$ satisfying $|a_j|\leq 1$ for all $j$ and $r:=\frac{p-\norma{\mu}}{2}\in\N_0$.
Let $I_\mu=\{i:1\leq i\leq n, \,  a_i=1\}\cup\{i:n+1\leq i\leq 2n,\, a_{i-n}=-1\}$.
The set $I_\mu$ has $p-2r$ elements.
For $I=\{i_1,\dots,i_p\}$ with $1\leq i_1<\dots<i_p\leq 2n$, it is a simple matter to check that $w_I$ is a weight vector with weight $\mu$ if and only if $I$ has $p$ elements, $I_\mu\subset I$ and $I$ has the property  that $j\in I\smallsetminus I_\mu \iff j+n\in I\smallsetminus I_\mu$ for $1\leq j\leq n$.
One can see that there are $\binom{n-p+2r}{r}$ choices for $I$.
Hence $ m_{\bigwedge^p(\C^{2n})}(\mu) = \binom{n-p+2r}{r}$.
From \eqref{eqCn:extremereps}, we conclude that $m_{\pi_{\omega_p}}(\mu) = m_{\bigwedge^p(\C^{2n})}(\mu) - m_{\bigwedge^{p-2}(\C^{2n})}(\mu) = \binom{n-p+2r}{r} - \binom{n-p+2+2r}{r}$ and \eqref{eqCn:multip(p)} is proved. 
\end{proof}

We next consider Step~3, namely, a multiplicity formula for $\sigma_{k,p}$.

\begin{lemma}\label{lemCn:multip(sigma_kp)}
Let $n\geq2$, $\mathfrak g=\spp(n,\C)$, $k\geq0$, $1\leq p<n$, and $\mu\in P(\mathfrak g)$.
If $r(\mu):=(k+p-\norma{\mu})/2$ is a non-negative integer, then 
	\begin{align*}
		m_{\sigma_{k,p}}(\mu)
		&= \sum_{t=0}^{\lfloor{p}/{2}\rfloor} \frac{n-p+1}{n-p+t+1}\binom{n-p+2t}{t}\sum_{\beta=0}^{p-2t} 2^{p-2t-\beta} \binom{n-\contador (\mu)}{\beta} \binom{\contador (\mu)}{p-2t-\beta}  \\
		&\qquad \sum_{\alpha=0}^\beta \binom{\beta}{\alpha} \binom{r(\mu)-p+\alpha+t+n-1}{n-1},
	\end{align*}
and $m_{\sigma_{k,p}}(\mu)=0$ otherwise. 
\end{lemma}

\begin{proof}
Write $r=r(\mu)$ and $\ell=\contador (\mu)$. 
We may assume that $\mu$ is dominant, thus $\mu=\sum_{j=1}^{n-\ell} a_j\varepsilon_j$ with $a_1\geq \dots \geq a_{n-\ell}>0$ since it has $\ell$ zero-coordinates.
In order to use \eqref{eq:multiptensor}, by Lemma~\ref{lemCn:extreme}, we write the set of weights of $\pi_{\omega_p}$ as 
$$
	\mathcal P(\pi_{\omega_p}) :=
	\bigcup_{t=0}^{\lfloor {p}/{2}\rfloor}
	\;\bigcup_{\beta=0}^{p-2t}
	\;\bigcup_{\alpha=0}^{\beta}
	\;\mathcal P_{t,\beta,\alpha}^{(p)}
$$
	where
\begin{equation}\label{eq:calP}
	\mathcal P_{t,\beta,\alpha}^{(p)} =
	\left\{
	\sum_{h=1}^{p-2t} b_h\varepsilon_{i_h}:
	\begin{array}{l}
	i_1<\dots<i_\beta\leq n-\ell< i_{\beta+1}<\dots<i_{p-2t} \\
	b_j=\pm1\quad \forall j,\quad \#\{1\leq j\leq \beta: b_j=1\}=\alpha
	\end{array}
	\right\}.
\end{equation}
A weight $\eta\in\mathcal P_{t,\beta,\alpha}^{(p)}$ has all entries in  $\{0,\pm 1\}$ and satisfies $\norma{\eta}=p-2t$, thus $m_{\pi_{\omega_p}}(\eta)=\frac{n-p+1}{n-p+t+1}\binom{n-p+2t}{t}$ by \eqref{eqCn:multip(p)}.
It is a simple matter to check that
\begin{equation}\label{eq:card(P)}
\# \mathcal P_{t,\beta,\alpha}^{(p)} = 2^{p-2t-\beta} \binom{n-\ell}{\beta}\binom{\beta}{\alpha} \binom{\ell}{p-2t-\beta}.
\end{equation}

From \eqref{eq:multiptensor}, since the triple union above is disjoint, we obtain that
	\begin{align*}
		m_{\sigma_{k,p}}(\mu)
		&= \sum_{t=0}^{\lfloor {p}/{2}\rfloor}\;\sum_{\beta=0}^{p-2t}\; \sum_{\alpha=0}^{\beta} \; \sum_{\eta\in \mathcal P_{t,\beta,\alpha}^{(p)}}
		m_{\pi_{k\varepsilon_1}}(\mu-\eta) \;m_{\pi_{\omega_p}}(\eta) .
\end{align*}
One has that $\norma{\mu-\eta} = (k+p-2r) +(\beta-\alpha)-\alpha + (p-2t-\beta) = k-2(r+t+\alpha-p)$ for every $\eta \in \mathcal P_{t,\beta,\alpha}^{(p)}$.
If $r\notin \N_{0}$, \eqref{eqCn:multip(k)} forces $m_{\pi_{k\varepsilon_1}}(\mu-\eta)=0$ for all $\eta \in \mathcal P_{t,\beta,\alpha}^{(p)}$, consequently $m_{\sigma_{k,p}}(\mu)=0$.  
Otherwise, 
\begin{align*}
m_{\sigma_{k,p}}(\mu)
		&= \sum_{t=0}^{\lfloor {p}/{2}\rfloor}\;\sum_{\beta=0}^{p-2t}\; \sum_{\alpha=0}^{\beta} \;
		\binom{r+t+\alpha-p+n-1}{n-1} \;\frac{n-p+1}{n-p+t+1}\; \binom{n-p+2t}{t} \; \# \mathcal P_{t,\beta,\alpha}^{(p)}
	\end{align*}
by Lemma~\ref{lemCn:extreme} . 
The proof is complete by \eqref{eq:card(P)}. 
\end{proof}

Theorem~\ref{thmCn:multip(k,p)} follows by replacing the multiplicity formula given in Lemma~\ref{lemCn:multip(sigma_kp)} into \eqref{eq:virtualring(sigma)}.

\section{Type D}\label{secDn:multip(k,p)}

We now consider type $\tipo D_n$, that is, $\mathfrak g=\so(2n,\C)$ and $G=\SO(2n)$. 
We assume that $n\geq2$, so the non-simple case $\mathfrak g=\so(4,\C)\simeq \sll(2,\C)\oplus \sll(2,\C)$ is also considered.

Since $G$ is not simply connected and has a fundamental group of order $2$, the lattice of $G$-integral weights $P(G)$ is strictly included  with index $2$ in the weight space $P(\mathfrak g)$.
Consequently, a dominant weight $\lambda$ in $P(\mathfrak g)\smallsetminus P(G)$ corresponds to a representation $\pi_{\lambda}$ of $\Spin(2n)$, which does not descend to a representation of  $G=\SO(2n)$.

In this case, for all $k\geq0$ and $1\leq p\leq n-2$, we have that 
\begin{align}\label{eqDn:pi_kp}
	\pi_{k,p}	&=\pi_{k \omega_1+\omega_{p}},&
	\pi_{k,n-1} &= \pi_{k\omega_1+\omega_{n-1}+\omega_n},&
	\pi_{k,n} &= \pi_{k\omega_1+2\omega_{n-1}}\oplus \pi_{k\omega_1+2\omega_{n}}.
\end{align}
Each of them descends to a representation of $G$ and its multiplicity formula is established in Theorem~\ref{thmDn:multip(k,p)}.
The remaining cases $\pi_{k \omega_1+\omega_n-1}$ and $\pi_{k \omega_1+\omega_n}$, are spin representations. 
Their multiplicity formulas were obtained in \cite[Lem.~4.2]{BoldtLauret-onenormDirac} and are stated in Theorem~\ref{thmDn:multip(spin)}.

\begin{theorem}\label{thmDn:multip(k,p)}
Let $\mathfrak g=\so(2n,\C)$ and $G=\SO(2n)$ for some $n\geq2$ and let $k\geq0$, $1\leq p\leq n$ integers.
For $\mu\in P(G)$, if $r(\mu):=(k+p-\norma{\mu})/2$ is a non-negative integer, then 
	\begin{align*}
		m_{\pi_{k,p}}(\mu)
		&= \sum_{j=1}^{p} (-1)^{j-1}  \sum_{t=0}^{\lfloor\frac{p-j}{2}\rfloor} \binom{n-p+j+2t}{t}  \sum_{\beta=0}^{p-j-2t} 2^{p-j-2t-\beta} \binom{n-\contador (\mu)}{\beta} \binom{\contador (\mu)}{p-j-2t-\beta}  \\
		&\quad \sum_{\alpha=0}^\beta \binom{\beta}{\alpha} \sum_{i=0}^{j-1} \binom{r(\mu)-i-p+\alpha+t+j+n-2}{n-2},
	\end{align*}
	and $m_{\pi_{k,p}}(\mu)=0$ otherwise.
Furthermore, $m_{\pi_{k,p}}(\mu)=0$ for every $\mu\in P(\mathfrak g)\smallsetminus P(G)$. 
\end{theorem}

\begin{theorem}\label{thmDn:multip(spin)} 
Let $\mathfrak g=\so(2n,\C)$ and $G=\SO(2n)$ for some $n\geq2$ and let $k\geq0$ an integer. 
Let $\mu\in P(\mathfrak g)\smallsetminus P(G)$. Write $r(\mu)= k+\frac{n}{2}- \norma{\mu}$, then 
	\begin{align*}
		m_{\pi_{k\omega_1+\omega_{n}}}(\mu)
		&= \begin{cases}
			\binom{r(\mu)+n-2}{n-2} &\text{ if }r(\mu)\geq0 \text{ and } \op{neg}(\mu)\equiv r(\mu)\pmod 2, \\
			0&\text{ otherwise},
		\end{cases} 
		\\
		m_{\pi_{k\omega_1+\omega_{n-1}}}(\mu)
		&= \begin{cases}
			\binom{r(\mu)+n-2}{n-2} &\text{ if }r(\mu)\geq0 \text{ and } \op{neg}(\mu)\equiv r(\mu)+1\pmod 2, \\
			0&\text{ otherwise},
		\end{cases} 
	\end{align*}
	where $\op{neg}(\mu)$ stands for the number of negative entries of $\mu$. 
Furthermore,  $m_{\pi_{k\omega_1+\omega_{n-1}}}(\mu) = m_{\pi_{k\omega_1+\omega_{n}}}(\mu) =0$ for every $\mu\in P(G)$.
\end{theorem}

The proof of Theorem~\ref{thmDn:multip(k,p)} will follow the steps from Section~\ref{sec:strategy}. 
Let us first set the necessary elements introduced in Notation~\ref{notacion}. 
Define
$\mathfrak h=
\left\{
\diag\left(
\left[\begin{smallmatrix}0&\theta_1\\ -\theta_1&0\end{smallmatrix}\right]
, \dots,
\left[\begin{smallmatrix}0&\theta_n\\ -\theta_n&0\end{smallmatrix}\right]
\right):
\theta_i\in\C \;\forall\,i
\right\}
$
and
$
\varepsilon_i\big(\diag\left(
\left[\begin{smallmatrix}0&\theta_1\\ -\theta_1&0\end{smallmatrix}\right]
, \dots,
\left[\begin{smallmatrix}0&\theta_n\\ -\theta_n&0\end{smallmatrix}\right]
\right)\big)=\theta_i
$
for each $1\leq i\leq n$. 
Thus $\Sigma^+(\mathfrak g,\mathfrak h)=\{\varepsilon_i\pm\varepsilon_j: i<j\}$,  
\begin{align*}
P(\mathfrak g) &= \{\textstyle \sum_i a_i\varepsilon_i: a_i\in\Z\,\forall i, \text{ or } a_i-1/2\in\Z\,\forall i\},&
P(G)&=\Z\varepsilon_1\oplus\dots\oplus\Z\varepsilon_{n}, \\
P^{{+}{+}}(\mathfrak g) &=\left\{\textstyle\sum_{i}a_i\varepsilon_i \in P(\mathfrak g) :a_1\geq  \dots\geq a_{n-1}\geq |a_n|\right\},&
P^{{+}{+}}(G)&= P^{{+}{+}}(\mathfrak g)\cap P(G).
\end{align*}
It is now clear that $P(G)$ has index $2$ in $P(\mathfrak g)$.

The multiplicity formulas in type $\tipo D_n$ for the extreme representations in Step~2 are already determined. 
A proof can be found in \cite[Lem.~3.2]{LMR-onenorm}.

\begin{lemma}\label{lemDn:extremereps}
	Let $n\geq2$, $\mathfrak g=\so(2n,\C)$, $G=\SO(2n)$, $k\geq0$ and $1\leq p\leq n$.
	For $\mu=\sum_{j=1}^n a_j\varepsilon_j\in P(G)$, we have that 
	\begin{align}
		m_{\pi_{k\omega_1}}(\mu) = &\;m_{\pi_{k\varepsilon_1}}(\mu) =
		\begin{cases}
			\binom{r(\mu)+n-2}{n-2} & \text{ if }\, r(\mu):=\frac{k-\norma{\mu}}{2} \in\N_0,\\
			0 & \text{ otherwise,}
		\end{cases}
		\label{eqDn:multip(k)} \\
		m_{\pi_{\widetilde \omega_p}}(\mu) =&
		\begin{cases}
			\binom{n-p+2r(\mu)}{r(\mu)} & \text{if }\, r(\mu):=\frac{p-\norma{\mu}}{2}\in \N_{0}  \text{ and }  |a_j|\leq1\;\forall\,j,\\
			0&\text{otherwise.} 
		\end{cases}
		\label{eqDn:multip(p)}
	\end{align}
\end{lemma}

\begin{lemma}\label{lemDn:multip(sigma_kp)}
	Let $n\geq2$, $\mathfrak g=\so(2n,\C)$, $G=\SO(2n)$, $k\geq0$, $1\leq p\leq n-1$, and $\mu\in P(G)$. 
	Write $r(\mu)=(k+p-\norma{\mu})/2$.
	If $r(\mu)$ is a non-negative integer, then
	\begin{align*}
		m_{\sigma_{k,p}}(\mu)
		= &\sum_{t=0}^{\lfloor{p}/{2}\rfloor} \binom{n-p+2t}{t}\sum_{\beta=0}^{p-2t} 2^{p-2t-\beta} \binom{n-\contador (\mu)}{\beta} \binom{\contador (\mu)}{p-2t-\beta}\\
		&\qquad \sum_{\alpha=0}^\beta \binom{\beta}{\alpha} \binom{r(\mu) -p+\alpha+t+n-2}{n-2},
	\end{align*}
	and $m_{\sigma_{k,p}}(\mu)=0$ otherwise. 
\end{lemma}

\begin{proof}
We will omit several details in the rest of the proof since it is very similar to the one of Lemma~\ref{lemCn:multip(sigma_kp)}. 
Write $r=(k+p-\norma{\mu})/2$ and $\ell=\contador (\mu)$. 
We assume that $\mu$ is dominant. 
Lemma~\ref{lemDn:extremereps} implies that the set of weights of $\pi_{\widetilde \omega_p}$ is
$
	\mathcal P(\pi_{\widetilde \omega_p}) :=
	\bigcup_{t=0}^{\lfloor {p}/{2}\rfloor}
	\;\bigcup_{\beta=0}^{p-2t}
	\;\bigcup_{\alpha=0}^{\beta}
	\;\mathcal P_{t,\beta,\alpha}^{(p)},
$
with $\mathcal P_{t,\beta,\alpha}^{(p)}$ as in \eqref{eq:calP}.

One has that $\norma{\mu-\eta} = k-2(r+t+\alpha-p)$ for any $\eta\in \mathcal P_{t,\beta,\alpha}^{(p)}$.
Hence, \eqref{eq:multiptensor} and Lemma~\ref{lemDn:extremereps} imply $m_{\sigma_{k,p}}(\mu)=0$ if $r\notin\N_0$ and 
\begin{align*}
	m_{\sigma_{k,p}}(\mu)
		= &\sum_{t=0}^{\lfloor {p}/{2}\rfloor}\;\sum_{\beta=0}^{p-2t}\; \sum_{\alpha=0}^{\beta} \;
		\binom{r+t+\alpha-p+n-2}{n-2} \;\binom{n-p+2t}{t} \; \# \mathcal P_{t,\beta,\alpha}^{(p)}
	\end{align*}
otherwise.  
The proof follows by \eqref{eq:card(P)}. 
\end{proof}

Theorem~\ref{thmDn:multip(k,p)} then follows by substituting in \eqref{eq:virtualring(sigma)} the multiplicity formula in Lemma~\ref{lemDn:multip(sigma_kp)}. 

\begin{remark}
By Definition~\ref{def:pi_kp}, $\pi_{k,n}$ in type $\tipo D_n$ is the only case where $\pi_{k,p}$ is not irreducible. 
We have that $\pi_{k,n}= \pi_{k\omega_1+\widetilde \omega_{n}} \oplus \pi_{k\omega_1+\widetilde \omega_{n}-2\varepsilon_n} = \pi_{k\omega_1+2\omega_{n-1}} \oplus \pi_{k\omega_1+2\omega_{n}}$ for every $k\geq0$. 
One can obtain the corresponding multiplicity formula for each of these irreducible constituents from Theorem~\ref{thmDn:multip(k,p)} by proving the following facts.
	If $\mu\in P(G)$ satisfies $\norma{\mu}=k+n$, then $m_{\pi_{k\omega_1+2\omega_n}}(\mu) =  m_{\pi_{k,n}}(\mu)$ and $m_{\pi_{k\omega_1+2\omega_{n-1}}}(\mu) = 0$ or $m_{\pi_{k\omega_1+2\omega_n}}(\mu) =  0$ and $m_{\pi_{k\omega_1+2\omega_{n-1}}}(\mu) = m_{\pi_{k,n}}(\mu)$ according $\mu$ has an even or odd number of negative entries respectively. 
	Furthermore, if $\mu\in P(G)$ satisfies $\norma{\mu} <k+n$, then $m_{\pi_{k\omega_1+2\omega_n}}(\mu) = m_{\pi_{k\omega_1+2\omega_{n-1}}}(\mu) = {m_{\pi_{k,n}}(\mu)}/{2}$.
\end{remark}

\section{Type B}\label{secBn:multip(k,p)}

We now consider $\mathfrak g=\so(2n+1,\C)$ and $G=\SO(2n+1)$, so $\mathfrak g$ is of type $\tipo B_n$. 
The same observation in the beginning of Section~\ref{secDn:multip(k,p)} is valid in this case.
Namely, a weight in $P^{{+}{+}}(\mathfrak g)  \smallsetminus P^{{+}{+}}(G)$ induces an irreducible representation of $\Spin(2n+1)$ which does not descend to $G$.

For any $k\geq0$ and $1\leq p\leq n-1$, we have that
\begin{align}
\pi_{k,p} &=\pi_{k \omega_1+\omega_{p}},&
\pi_{k,n} &=\pi_{k\omega_1+2\omega_{n}}.
\end{align}
All of them descend to representations of $G$. 
The corresponding multiplicity formula is in Theorem~\ref{thmBn:multip(k,p)} and the remaining case, $\pi_{k\omega_1+\omega_n}$ for $k\geq0$, is considered in  Theorem~\ref{thmBn:multip(spin)}.

\begin{theorem}\label{thmBn:multip(k,p)}
Let $\mathfrak g=\so(2n+1)$, $G=\SO(2n+1)$ for some $n\geq2$ and let $k\geq0$,  $1\leq p\leq n$ integers.
For $\mu\in P(G)$, write $r(\mu)=k+p-\norma{\mu}$, then
	\begin{align*}
		m_{\pi_{k,p}}(\mu)
		= &\sum_{j=1}^{p} (-1)^{j-1}  \sum_{t=0}^{\lfloor\frac{p-j}{2}\rfloor} \binom{n-p+j+2t}{t}  \\
		&\qquad \sum_{\beta=0}^{p-j-2t} 2^{p-j-2t-\beta} \binom{n-\contador (\mu)}{\beta} \binom{\contador (\mu)}{p-j-2t-\beta}  \\
		&\qquad \sum_{\alpha=0}^\beta \binom{\beta}{\alpha} \sum_{i=0}^{j-1} \binom{\lfloor\frac{r(\mu)}{2}\rfloor-i-p+j+\alpha+t+n-1}{n-1}\\
		&+\sum_{j=1}^{p-1} (-1)^{j-1}  \sum_{t=0}^{\lfloor\frac{p-j-1}{2}\rfloor} \binom{n-p+j+2t+1}{t}  \\
		&\qquad \sum_{\beta=0}^{p-j-2t-1} 2^{p-j-2t-\beta-1} \binom{n-\contador (\mu)}{\beta} \binom{\contador (\mu)}{p-j-2t-\beta-1}  \\
		&\qquad \sum_{\alpha=0}^\beta \binom{\beta}{\alpha} \sum_{i=0}^{j-1} \binom{\lfloor\frac{r(\mu)+1}{2}\rfloor -i-p+j+\alpha+t+n-1}{n-1}. 	
	\end{align*}
Furthermore, $m_{\pi_{k,p}}(\mu)=0$ for all $\mu\in  P(\mathfrak g)\smallsetminus P(G)$. 
\end{theorem}

\begin{remark}
Notice that, in Theorem~\ref{thmBn:multip(k,p)}, $m_{\pi_{k,p}}(\mu)=0$ if $r(\mu)<0$ because of the convention $\binom{b}{a}=0$ if $b<a$. 
\end{remark}

We will omit most of details since this case is very similar to the previous ones, specially to type $\tipo D_n$.
According to Notation~\ref{notacion}, we set 
$\mathfrak h=
\left\{
\diag\left(
\left[\begin{smallmatrix}0&\theta_1\\ -\theta_1&0\end{smallmatrix}\right]
, \dots,
\left[\begin{smallmatrix}0&\theta_n\\ -\theta_n&0\end{smallmatrix}\right],0
\right):
\theta_i\in\C \;\forall\,i
\right\}$, 
$
\varepsilon_i\big(\diag\left(
\left[\begin{smallmatrix}0&\theta_1\\ -\theta_1&0\end{smallmatrix}\right]
, \dots,
\left[\begin{smallmatrix}0&\theta_n\\ -\theta_n&0\end{smallmatrix}\right],0
\right)\big)=\theta_i
$
for each $1\leq i\leq n$, $\Sigma^+(\mathfrak g,\mathfrak h)=\{\varepsilon_i\pm\varepsilon_j: i<j\}\cup\{\varepsilon_i\}$, 
\begin{align*}
	P(\mathfrak g) &= \{\textstyle \sum_i a_i\varepsilon_i: a_i\in\Z\,\forall i, \text{ or } a_i-1/2\in\Z\,\forall i\},&
	P(G)&=\Z\varepsilon_1\oplus\dots\oplus\Z\varepsilon_{n}, \\
	P^{{+}{+}}(\mathfrak g) &=\left\{\textstyle\sum_{i}a_i\varepsilon_i \in P(\mathfrak g) :a_1\geq a_{2}\geq \dots\geq a_{n}\geq0\right\},&
	P^{{+}{+}}(G)&= P^{{+}{+}}(\mathfrak g)\cap P(G).
\end{align*}
It is well known that (see \cite[Exercises~IV.10 and V.8]{Knapp-book-beyond})
\begin{align}\label{eqBn:extremereps}
\op{Sym}^k(\C^{2n+1}) &\simeq \pi_{k\omega_1}\oplus \op{Sym}^{k-2}(\C^{2n+1}),&
\pi_{\widetilde \omega_p}& \simeq \textstyle \bigwedge^p(\C^{2n+1}), 
\end{align}
where $\C^{2n+1}$ denotes the standard representation of $\mathfrak g$. 
Actually, $\pi_{k\omega_1}$ can be realized inside $\op{Sym}^k(\C^{2n+1})$ as the subspace of harmonic homogeneous polynomials of degree $k$.

\begin{lemma}\label{lemBn:extremereps}
	Let $n\geq2$, $\mathfrak g=\so(2n+1,\C)$, $G=\SO(2n+1)$, $k\geq0$ and $1\leq p\leq n$.
	For $\mu=\sum_{j=1}^n a_j\varepsilon_j\in P(G)$, we have that
\begin{align}
m_{\pi_{k\omega_1}}(\mu) = &\;m_{\pi_{k\varepsilon_1}}(\mu) =
\tbinom{r(\mu)+n-1}{n-1}  \quad \text{ where } r(\mu)=\lfloor \tfrac{k-\norma{\mu}}{2}\rfloor, 		
		\label{eqBn:multip(k)} \\
m_{\pi_{\widetilde \omega_p}}(\mu) =&
\begin{cases}
\binom{n-p+r(\mu)}{\lfloor {r(\mu)}/{2}\rfloor} & \text{if }\, |a_j|\leq1\;\forall\,j, \\
0&\text{otherwise,} 
\end{cases}\qquad\text{where $r(\mu)=p-\norma{\mu}$}. 
		\label{eqBn:multip(p)}
	\end{align}
\end{lemma}

\begin{proof} 
Let $\mathcal P_k$ be the space of complex homogeneous polynomials of degree $k$ in the variables $x_1,\dots,x_{2n+1}$. 
Set $f_j=x_{2j-1}+ix_{2j}$ and $g_{j}= x_{2j-1}-ix_{2j}$ for $1\leq j\leq n$.
One can check that the polynomials 
$
f_1^{k_1}\dots f_n^{k_n} g_1^{l_1}\dots g_{n}^{l_{n}}x_{2n+1}^{k_0}
$ 
with $k_0,\dots,k_n,l_1,\dots,l_n$ non-negative integers satisfying that $\sum_{j=0}^{n} k_j+\sum_{j=1}^{n} l_j=k$ form a basis of $\mathcal P_k$ given by weight vectors, each of them of weight $\mu=\sum_{j=1}^n (k_j-l_{j})\varepsilon_j$.
Notice that the number $k_0$ does not take part of $\mu$.

Consequently, $m_{\pi_{\mathcal P_k}}(\mu)$ for $\mu=\sum_{j=1}^n a_j\varepsilon_j$ is the number of tuples $(k_0,\dots,k_{n}, l_1,\dots,l_{n})\in \N_0^{2n+1}$ satisfying that $a_j=k_j-l_{j}$ for all $1\leq j\leq n$ and 
\begin{equation}\label{eqBn:conditionweightP_k}
\sum_{j=0}^{n} k_j+\sum_{j=1}^{n} l_j=k.
\end{equation}
Note that \eqref{eqBn:conditionweightP_k} implies $k-\norma{\mu}-k_0=2s$ for some integer $s\geq0$.

We fix an integer $s$ satisfying $0\leq s\leq r:=\lfloor (k-\norma{\mu})/2\rfloor $.
Set $k_0=k-\norma{\mu}-2s\geq0$. 
As in the proof of Lemma~\ref{lemCn:extreme}, the number of $(k_1,\dots,k_n,l_1,\dots,l_n)\in \N_0^{2n}$ satisfying that $a_j=k_j-l_j$ for all $1\leq j\leq n$ and \eqref{eqBn:conditionweightP_k} is equal to $\binom{s+n-1}{n-1}$. 
Hence,  
\begin{equation*}
	m_{\mathcal P_k}(\mu) = \sum_{s=0}^{r} 
	\binom{s+n-1}{n-1}=  \binom{r+n}{n}.
\end{equation*}
The second equality is well known. 
It may be proven by showing that both sides are the $r$-term of the generating function $(1-z)^{-(n+1)}$. 
From \eqref{eqBn:extremereps} we conclude that $m_{\pi_{k\varepsilon_1}}(\mu) = m_{{\mathcal P}_k}(\mu) - m_{{\mathcal P}_{k-2}}(\mu) = \binom{r+n}{n}- \binom{r-1+n}{n} = \binom{r+n-1}{n-1}$.

We have that $\pi_{\widetilde\omega_p}\simeq \bigwedge^p(\C^{2n+1})$ by \eqref{eqBn:extremereps}.
By setting $v_j=e_{2j-1}-i e_{2j}$, $v_{j+n}=e_{2j-1}+i e_{2j}$ and $v_{2n+1}=e_{2n+1}$,
one obtains that the vectors $w_I:=v_{i_1}\wedge \dots\wedge v_{i_p}$ for $I=\{i_1,\dots,i_p\}$ satisfying $1\leq i_1<\dots<i_p\leq 2n+1$, 
form a basis of $\bigwedge^p(\C^{2n+1})$. Furthermore, $w_I$ is a weight vector of weight 
$\mu=\sum_{j=1}^n a_j\varepsilon_j$ given by \eqref{eq:weight_exteriorCn}.
Note that the condition of $2n+1$ being or not in $I$ does not influence on $\mu$. 

Hence, $\mu=\sum_j a_j\varepsilon_j$ is a weight of $\bigwedge^p(\C^{2n+1})$ if and only if $|a_j|\leq 1$ for all $j$ and $p-\norma{\mu}\geq0$.
Proceeding as in Lemma~\ref{lemCn:extreme}, by writing  $s=\lfloor \frac{p-\norma{\mu}}{2} \rfloor\geq0$, the multiplicity of $\mu$ is $\binom{n-p+2s}{s}$ if $p-\norma{\mu}$ is even and $\binom{n-p+2s+1}{s}$ if $p-\norma{\mu}$ is odd.
\end{proof}

\begin{theorem}\label{thmBn:multip(spin)}
	Let $\mathfrak g=\so(2n+1,\C)$ and $G=\SO(2n+1)$ for some $n\geq2$ and let $k\geq0$ an integer. 
	Let $\mu\in P(\mathfrak g)\smallsetminus P(G)$. Write $r(\mu)=k+\frac{n}{2}-\norma{\mu}$, then 
	\begin{equation}\label{eqBn:multip(spin)}
	m_{\pi_{k\omega_1+\omega_n}}(\mu) =\binom{r(\mu)+n-1}{n-1}.
	\end{equation}
	Furthermore, $m_{k\omega_1+\omega_n}(\mu)=0$ for all $\mu\in P(G)$. 
\end{theorem}

\begin{proof}
	This proof is very similar to  \cite[Lem.~4.2]{BoldtLauret-onenormDirac}. 
	The assertion $m_{k\omega_1+\omega_n}(\mu)=0$ for every $\mu\in P(G)$ is clear since any weight of $\pi_{k\omega_1+\omega_n}$ is equal to the highest weight $k\omega_1+\omega_n$ minus a sum of positive roots, which clearly lies in $P(\mathfrak g)\smallsetminus P(G)$.

	Let $\mu\in P(\mathfrak g)\smallsetminus P(G)$. 
	We may assume that $\mu$ is dominant, thus $\mu=\frac{1}{2}\sum_{i=1}^n a_i\varepsilon_i$ with $a_1\geq \dots \geq a_n \geq1$ odd integers.
	One has that 
	\begin{align}\label{eq:fusionrule_spin}
		\pi_{k \omega_1}\otimes \pi_{\omega_n} \simeq \pi_{k \omega_1+\omega_n} \oplus \pi_{(k-1) \omega_1+\omega_n}
	\end{align}
	for any $k\geq1$.
	Indeed, it follows immediately by applying the formula in \cite[Exercise~V.19]{Knapp-book-beyond} since in its sum over the weights of $\pi_{\omega_n}$, the only non-zero terms are attained at the weights $\omega_n$ and $\omega_n-\omega_1$. 
	
	It is well known that the set of weights of $\pi_{\omega_n}$ is $\mathcal{P}(\pi_{\omega_n}) :=\{ \frac{1}{2}\sum_{i=1}^n b_i\varepsilon_i: |b_i|=1\}$ and $m_{\pi_{\omega_n}}(\nu)=1$ for all $\nu\in \mathcal{P}(\pi_{\omega_n})$ (see for instance \cite[Exercise V.35]{Knapp-book-beyond}).

	We proceed now to prove \eqref{eqBn:multip(spin)} by induction on $k$. 
	It is clear for $k=0$ by the previous paragraph. 
	Suppose that it holds for $k-1$. 
	By this assumption and \eqref{eq:fusionrule_spin}, we obtain that
	\begin{equation}\label{eqBn:multip(tensorspin)}
	m_{\pi_{k\omega_1+\omega_n}}(\mu)= m_{\pi_{k\omega_1}\otimes \pi_{\omega_n}}(\mu)-m_{\pi_{(k-1)\omega_1+\omega_n}}(\mu) = 
	m_{\pi_{k\omega_1}\otimes \pi_{\omega_n}}(\mu)- \binom{r+n-2}{n-1},
	\end{equation}
	where $r=k+\frac{n}{2}-\norma{\mu}$.
	It only remains to prove that $m_{\pi_{k\omega_1}\otimes \pi_{\omega_n}}(\mu)=\binom{r+n-1}{n-1}+\binom{r+n-2}{n-1}$.
	
	Similarly to \eqref{eq:multiptensor}, we have that $m_{\pi_{k\omega_1}\otimes \pi_{\omega_n}}(\mu)=\sum_{\eta\in\mathcal{P}(\pi_{\omega_n})} m_{\pi_{k\omega_1}}(\mu-\eta)$.
	Since $\mu$ is dominant, for any $\eta=\frac{1}{2}\sum_{i=1}^n b_i\varepsilon_i\in\mathcal{P}(\pi_{\omega_n})$, it follows that
	$$
	\norma{\mu-\eta}=
	\frac{1}{2}\sum_{i=1}^n (a_i-b_i)  = \norma{\mu}+\frac{n}{2}-\ell_1(\eta)=
	k-r+n-\ell_1(\eta),
	$$
	where $\ell_1(\eta)=\#\{1\leq i\leq n: b_i=1\}$.  
	By Lemma \ref{lemBn:extremereps}, $m_{\pi_{k \omega_1}}(\mu-\eta)\neq0$ only if $r +\ell_1(\eta)-n\geq0$.
	For each integer $\ell_1$ satisfying $n-r\leq\ell_1\leq n$, there are $\binom{n}{\ell_1}$ weights $\eta\in \mathcal{P}(\pi_{\omega_n})$ such that $\ell_1(\eta)=\ell_1$. 
	On account of the above remarks, 
	\begin{align}\label{eq:multiptensor_spin}
		m_{\pi_{k\omega_1}\otimes \pi_{\omega_n}}(\mu)=& \sum_{\ell_1=n-r}^{n} \binom{\lfloor \frac{r+\ell_1-n}{2}\rfloor +n-1}{n-1} \binom{n}{\ell_1}= 
		\sum_{j=0}^{r} \binom{\lfloor \frac{r-j}{2}\rfloor +n-1}{n-1} \binom{n}{j}.
	\end{align}
	
	We claim that the last term in \eqref{eq:multiptensor_spin} equals $\binom{r+n-1}{n-1}+\binom{r+n-2}{n-1}$.
	Indeed, a simple verification shows that both numbers are the $r$-th term of the generating function $\frac{1+z}{(1-z)^n}$. 
	From \eqref{eqBn:multip(tensorspin)} and \eqref{eq:multiptensor_spin} we conclude that $m_{\pi_{k\omega_1+\omega_n}}(\mu)= 	\binom{r+n-1}{n-1}$	as asserted. 
\end{proof}

\begin{lemma}\label{lemBn:multip(sigma_kp)}
	Let $n\geq2$, $\mathfrak g=\so(2n+1,\C)$, $G=\SO(2n+1)$, $k\geq0$, $1\leq p<n$, and $\mu\in P(G)$. 
	Write $r(\mu)=k+p-\norma{\mu}$.
	Then
	\begin{align*}
		m_{\sigma_{k,p}}(\mu)
		= &\sum_{t=0}^{\lfloor{p}/{2}\rfloor} \binom{n-p+2t}{t}\sum_{\beta=0}^{p-2t} 2^{p-2t-\beta} \binom{n-\contador (\mu)}{\beta} \binom{\contador (\mu)}{p-2t-\beta}\\
		&\qquad \sum_{\alpha=0}^\beta \binom{\beta}{\alpha} \binom{\lfloor\frac{r(\mu)}{2}\rfloor-p+\alpha+t+n-1}{n-1}\\
		&+\sum_{t=0}^{\lfloor{(p-1)}/{2}\rfloor} \binom{n-p+1+2t}{t}\sum_{\beta=0}^{p-1-2t} 2^{p-1-2t-\beta} \binom{n-\contador (\mu)}{\beta} \binom{\contador (\mu)}{p-1-2t-\beta}\\
		&\qquad \sum_{\alpha=0}^\beta \binom{\beta}{\alpha} \binom{\lfloor\frac{r(\mu)+1}{2}\rfloor-p+\alpha+t+n-1}{n-1}.
	\end{align*}
\end{lemma}

\begin{proof}
Write $r=k+p-\norma{\mu}$ and $\ell=\contador (\mu)$ and assume $\mu$ dominant.
Define $\mathcal P_{t,\beta,\alpha}^{(p)}$ as in \eqref{eq:calP}.
From Lemma~\ref{lemBn:extremereps}, we deduce that the set of weights of $\pi_{\widetilde \omega_p}$ is
	$$
	\mathcal P(\pi_{\widetilde \omega_p}) :=
	\big(\bigcup_{t=0}^{\lfloor {p}/{2}\rfloor}
	\;\bigcup_{\beta=0}^{p-2t}
	\;\bigcup_{\alpha=0}^{\beta}
	\;\mathcal P_{t,\beta,\alpha}^{(p)} \big) \cup 
	\big(\bigcup_{t=0}^{\lfloor {p-1}/{2}\rfloor}
	\;\bigcup_{\beta=0}^{p-1-2t}
	\;\bigcup_{\alpha=0}^{\beta}
	\;\mathcal P_{t,\beta,\alpha}^{(p-1)}\big).
	$$
This fact and \eqref{eq:multiptensor} give 
	\begin{align*}
		m_{\sigma_{k,p}}(\mu)
		= &\sum_{t=0}^{\lfloor {p}/{2}\rfloor}\;\sum_{\beta=0}^{p-2t}\; \sum_{\alpha=0}^{\beta} \;
		\binom{\lfloor \frac{r}{2}\rfloor+t+\alpha-p+n-1}{n-1} \;\binom{n-p+2t}{t} \; \# \mathcal P_{t,\beta,\alpha}^{(p)} \\
		&  +\sum_{t=0}^{\lfloor {(p-1)}/{2}\rfloor}\;\sum_{\beta=0}^{p-1-2t}\; \sum_{\alpha=0}^{\beta} \;
		\binom{\lfloor \frac{r-1}{2}\rfloor+t+\alpha-p+n}{n-1} \;\binom{n-p+1+2t}{t} \; \# \mathcal P_{t,\beta,\alpha}^{(p-1)},
	\end{align*}
since $\norma{\mu-\eta}= k-r-2(t+\alpha-p)$ for all $\eta\in\mathcal P_{t,\beta,\alpha}^{(p)}$ and $\norma{\mu-\eta}= k-r-2(t+\alpha-p)-1$ for all $\eta\in\mathcal P_{t,\beta,\alpha}^{(p-1)}$.
The proof follows by \eqref{eq:card(P)}. 
\end{proof}

Lemmas~\ref{lem:step1} and \ref{lemBn:multip(sigma_kp)} complete the proof of Theorem~\ref{thmBn:multip(k,p)}.

\section{Type A}\label{secAn:multip(k,p)}
Type $\tipo A_n$ is the simplest case to compute the weight multiplicity formula of $\pi_{k,p}$.
Actually, it follows immediately by standard calculations using Young diagrams. 
We include this formula to complete the list of all classical simple Lie algebras. 

We consider in $\mathfrak g=\sll(n+1,\C)$, 
$
\mathfrak h =\{\diag\big(\theta_1,\dots,\theta_{n+1}\big) : \theta_i\in\C\;\forall\, i,\; \sum_{i=1}^{n+1}\theta_i=0\}.
$
We set $\varepsilon_i\big(\diag(\theta_1,\dots,\theta_{n+1})\big)= \theta_i$ for each $1\leq i\leq n+1$.
We will use the conventions of \cite[Lecture~15]{FultonHarris-book}.
Thus 
\begin{equation*}
\mathfrak h^* = \bigoplus_{i=1}^{n+1} \C\varepsilon_i / \langle \textstyle\sum\limits_{i=1}^{n+1}\varepsilon_i=0 \rangle,
\end{equation*}
the set of positive roots is
$\Sigma^+(\mathfrak g,\mathfrak h)=\{\varepsilon_i-\varepsilon_j: 1\leq i<j\leq n+1\}$, and the weight lattice is
$
P(\mathfrak g)=\bigoplus_{i=1}^{n+1} \Z\varepsilon_i / \langle \textstyle\sum\limits_{i=1}^{n+1}\varepsilon_i=0 \rangle.
$
By abuse of notation, we use the same letter $\varepsilon_i$ for the image of $\varepsilon_i$ in $\mathfrak h^*$. 
A weight $\mu=\sum_{i=1}^{n+1}a_i\varepsilon_i$ is dominant if $a_1\geq a_2\geq \dots \geq a_{n+1}$.

The representations having highest weights $\lambda=\sum_{i=1}^{n+1} a_i\varepsilon_i$ and $\mu=\sum_{i=1}^{n+1} b_i\varepsilon_i$ are isomorphic if and only if $a_i-b_i$ is constant, independent of $i$. 
Consequently, we can restrict to those $\lambda=\sum_{i=1}^{n+1} a_i\varepsilon_i$ with $a_{n+1}=0$. 
Then, 
$$
P^{++}(\mathfrak g)= \left\{ \textstyle\sum\limits_{i=1}^n a_i\varepsilon_i\in P(\mathfrak g): a_1\geq a_2\geq \dots \geq a_{n}\geq 0 \right\}.
$$

The corresponding fundamental weights are given by
$\omega_p=\varepsilon_{1} + \dots + \varepsilon_p$ for each $1\leq p\leq n$.

It is well known that, for $\lambda\in P^{++}(\mathfrak g)$ and $\mu$ a weight of $\pi_{\lambda}$, one can assume that $\mu=\sum_{i=1}^{n+1} a_i \varepsilon_i$ with $a_i\in \N_0$ for all $i$ and $\sum_{i=1}^{n+1} a_i=\norma{\lambda}$. 
%In other words, $\mu$ is a partition of $\norma{\lambda}$. 

\begin{theorem}\label{thmAn:multip(k,p)}
Let $\mathfrak g=\sll(n+1,\C)$ for some $n\geq1$ and let $k\geq0$,  $1\leq p\leq n$ integers. 
Let $\mu=\sum_{i=1}^{n+1} a_i \varepsilon_i\in P(\mathfrak g)$ with $a_i\in \N_0$ for all $i$ and $\sum_{i=1}^{n+1} a_i=k+p$. 
If $a_1+a_2+\dots +a_j\leq k+j$ for all $1\leq j\leq p$, then
	\begin{align*}
		m_{\pi_{k\omega_1+\omega_p}}(\mu)
		&= \binom{n- \contador (\mu)}{p-1},
	\end{align*}
	and $m_{\pi_{k,p}}(\mu)=0$ otherwise.
\end{theorem}
\begin{proof}
	The Young diagram corresponding to the representation $\pi_{k \omega_1+\omega_p}$ is the diagram with $p$ rows, having all length $1$, excepting the first one which has length $k+1$.
	It is well known that the multiplicity of the weight $\mu$ in this representation is equal to the number of ways one can fill its Young diagram with $a_1$ $1$'s, $a_2$ $2$'s, $\dots$, $a_{n+1}$ $(n+1)$'s, in such a way that the entries in the first row are non-decreasing and those in the first column are strictly increasing (see for instance \cite[\S15.3]{FultonHarris-book}).
	
	Consequently, the multiplicity of $\mu$ is equal to the number of ways of filling the first column. Since the first entry is uniquely determined, one has to choose $p-1$ different numbers for the rest of the entries. Hence, the theorem follows.
\end{proof}

\section{Concluding remarks}\label{sec:conclusions}

For a classical complex Lie algebra $\mathfrak g$, it has been shown a closed explicit formula for the weight multiplicities of a representation in any $p$-fundamental string, namely, any irreducible representation of $\mathfrak g$ having highest weight $k\omega_1+\omega_p$, for some integers $k\geq0$ and $1\leq p\leq n$. 
When $\mathfrak g$ is of type $\tipo A_n$, the proof was quite simple and the corresponding formula could be probably established from a more general result. 
In the authors' best knowledge, the obtained expressions of the weight multiplicities for types $\tipo B_n$, $\tipo C_n$ and $\tipo D_n$ are new, except for small values of $n$, probably $n\leq 3$.

Although the formulas in Theorem~\ref{thmCn:multip(k,p)}, \ref{thmDn:multip(k,p)} and \ref{thmBn:multip(k,p)} (types $\tipo C_n$, $\tipo D_n$ and $\tipo B_n$ respectively) look complicated and long, they are easily handled in practice. 
It is important to note that all sums are over (integer) intervals, without including any sum over partitions or permutations. 
Furthermore, there are only combinatorial numbers in each term. 
Consequently, it is a simple matter to implement them in a computer program, obtaining a very fast algorithm even when the rank $n$ of the Lie algebra is very large.

Moreover, for $p$ and a weight $\mu$ fixed, the formulas become a quasi-polynomial on $k$. 
This fact was already predicted and follows by the Kostant Multiplicity Formula, such as M.~Vergne pointed out to Kumar and Prasad in \cite{KumarPrasad14} (see also \cite{MeinrenkenSjamaar99}, \cite{Bliem10}).

For instance, when $\mathfrak g=\so(2n,\C)$ (type $\tipo D_n$), Theorem~\ref{thmDn:multip(k,p)} ensures that 
\begin{equation}
m_{\pi_{k\omega_1}}(\mu) =
\begin{cases}
\binom{\frac{k-\norma{\mu}}{2}+n-2}{n-2}
	&\text{if $k\geq\norma{\mu}$ and $k\equiv\norma{\mu} \pmod 2$,} \\
0 &\text{otherwise.}
\end{cases}
\end{equation}
Consequently, the generating function encoding the numbers $\{m_{\pi_{k\omega_1}}(\mu):k\geq0\}$ is a rational function.
Indeed, 
\begin{equation}
\sum_{k\geq0} m_{\pi_{k\omega_1}}(\mu) z^k = 
\sum_{k\geq0} m_{\pi_{(2k+\norma{\mu})\omega_1}}(\mu) z^{2k+\norma{\mu}} = \frac{z^{\norma{\mu}}}{(1-z^2)^{n-1}}. 
\end{equation}

From a different point of view, for fixed integers $k$ and $p$, the formulas are quasi-polynomials in the variables $\norma{\mu}$ and $\contador(\mu)$.

We end the article with a summary of past (and possible future) applications of multiplicity formulas in spectral geometry. 
We consider a locally homogeneous space $\Gamma\ba G/K$ with the (induced) standard metric, where $G$ is a compact semisimple Lie group, $K$ is a closed subgroup of $G$ and $\Gamma$ is a finite subgroup of the maximal torus $T$ of $G$. 
When $G=\SO(2n)$, $K=\SO(2n-1)$ and $\Gamma$ is cyclic acting freely on $G/K\simeq S^{2n-1}$, we obtain a \emph{lens space}.

In order to determine explicitly the spectrum of a (natural) differential operator acting on smooth sections of a (natural) vector bundle on $\Gamma\ba G/K$ (e.g.\ Laplace--Beltrami operator, Hodge--Laplace operator on $p$-form, Dirac operator), one has to calculate ---among other things--- numbers of the form $\dim V_\pi^\Gamma$ for $\pi$ in a subset of the unitary dual $\widehat G$ depending on the differential operator.  
Since $\Gamma\subset T$, $\dim V_\pi^\Gamma$ can be computed by counting the $\Gamma$-invariant weights in $\pi$ according to its multiplicity, so the problem is reduced to know $m_\pi(\mu)$.

At the moment, some weight multiplicity formulas have been successfully applied to the problem described above. 
The multiplicity formula for $\pi_{k\omega_1}$ in type $\tipo D_n$ (Lemma~\ref{lemDn:extremereps})  was used by Miatello, Rossetti and the first named author in \cite{LMR-onenorm} to determine the spectrum of the Laplace--Beltrami operator on a lens space. 
Furthermore, Corollary~\ref{cor:depending-one-norm-ceros} for type $\tipo D_n$ was shown in the same article (\cite[Lem.~3.3]{LMR-onenorm}) obtaining a characterization of lens spaces $p$-isospectral for all $p$ (i.e.\  their Hodge--Laplace operators on $p$-forms have the same spectra). 
Later, Boldt and the first named author considered in \cite{BoldtLauret-onenormDirac} the Dirac operator on odd-dimensional spin lens spaces.
In this work, it was obtained and used Theorem~\ref{thmDn:multip(spin)}, namely, the multiplicity formula for type $\tipo D_n$ of the spin representations $\pi_{k\omega_1+\omega_{n-1}}$ and $\pi_{k \omega_1+\omega_{n}}$.  

As a continuation of the study begun in \cite{LMR-onenorm}, Theorem~\ref{thmDn:multip(k,p)} was applied in the preprint \cite{Lauret-pspectralens} to determine explicitly every $p$-spectra of a lens space.
Here, as usual, $p$-spectrum stands for the spectrum of the Hodge--Laplace operator acting on smooth $p$-forms.  
The article \cite{Lauret-pspectralens} was the motivation to write the present paper.

The remaining formulas in the article may be used with the same goal. 
Actually, any application of the formulas for type $\tipo D_n$ can be translated to an analogue application for type $\tipo B_{n-1}$, working in spaces covered by $S^{2n-2}$ in place of $S^{2n-1}$ (cf.\ \cite[\S4]{IkedaTaniguchi78}).
This was partially done in \cite{Lauret-spec0cyclic}, by applying Lemma~\ref{lemBn:extremereps}.
The result extends \cite{LMR-onenorm} (for the Laplace--Beltrami operator) to even-dimensional lens orbifolds. 

A different but feasible application can be done for type $\tipo A_n$.
One may consider the complex projective space $P^n(\C)=\SU(n+1)/\op{S}(\U(n)\times\U(1))$.
However, more general representations must be used.
Indeed, in \cite{Lauret-spec0cyclic} was considered the Laplace--Beltrami operator and the representations involved had highest weights $k(\omega_1+\omega_n)$ for $k\geq0$.

Theorem~\ref{thmCn:multip(k,p)} (type $\tipo C_n$) does not have an immediate application since the spherical representations of the symmetric space $\Sp(n)/(\Sp(n-1)\times\Sp(1))$ have highest weight of the form $k\omega_2$ for $k\geq0$. 
Maddox~\cite{Maddox14} obtained a multiplicity formula for these representations. 
However, this expression it is not explicit enough to be applied in this problem. 
An exception was the case $n=2$, since in \cite{Lauret-spec0cyclic} was applied the closed multiplicity formula in \cite{CaglieroTirao04}. 
It is not know by the authors if there is a closed subgroup $K$ of $G=\Sp(n)$ such that the spherical representations of $G/K$ are $\pi_{k\omega_1}$ for $k\geq0$, that is, 
\begin{equation}
\{\pi\in \widehat G: V_\pi^K\simeq \op{Hom}_K(V_\pi,\C)\neq0\} = \{\pi_{k\omega_1}:k\geq0\}.
\end{equation}
In such a case, Theorem~\ref{thmCn:multip(k,p)} could be used.

\section*{Acknowledgments}
	The authors wish to thank the anonymous referee for carefully reading the article and giving them helpful comments.

\bibliographystyle{plain}

\end{document}